\documentclass[final]{siamart1116}

\usepackage{lipsum}
\usepackage{amsfonts}
\usepackage{graphicx}
\usepackage{epstopdf}
\usepackage{algorithmic}
\usepackage{enumerate}

\usepackage[T1]{fontenc}
\usepackage[utf8]{inputenc}
\ifpdf
  \DeclareGraphicsExtensions{.eps,.pdf,.png,.jpg}
\else
  \DeclareGraphicsExtensions{.eps}
\fi

\numberwithin{theorem}{section}

\newcommand{\TheTitle}{On the construction of converging hierarchies\\ for polynomial optimization\\ based on certificates of global positivity} 
\newcommand{\TheAuthors}{ }
\headers{Converging hierarchies for polynomial optimization}{\TheAuthors}

\title{{\TheTitle}\thanks{This work is partially funded by the DARPA Young Faculty Award, the Young Investigator Award of the AFOSR, the CAREER Award of the NSF, the Google Faculty Award, and the Sloan Fellowship.}}

\author{
  Amir Ali Ahmadi\thanks{Department of Operations Research and Financial Engineering, Princeton, NJ 08540
    (\email{a\_a\_a@princeton.edu}, \url{http://aaa.princeton.edu}).}
  \and
  Georgina Hall\thanks{Department of Operations Research and Financial Engineering, Princeton, NJ 08540 (\email{gh4@princeton.edu},
    \url{http://scholar.princeton.edu/ghall}).}
}

\usepackage{amsopn}

\newsiamthm{remark}{Remark}
\newsiamthm{example}{Example}

\ifpdf
\hypersetup{
  pdftitle={\TheTitle},
  pdfauthor={\TheAuthors}
}
\fi

\usepackage{color}

\newenvironment{gh}{\color{black}}

\def\aaa#1{{\color{black}#1}}

\begin{document}

\maketitle

% REQUIRED
\begin{abstract}
In recent years, techniques based on convex optimization and real algebra that produce converging hierarchies of lower bounds for {\gh polynomial minimization problems} have gained much popularity. At their heart, these hierarchies rely crucially on Positivstellens\"atze from the late 20th century (e.g., due to Stengle, Putinar, or Schm\"udgen) that certify positivity of a polynomial on an arbitrary {\gh closed }basic semialgebraic set. In this paper, we show that such hierarchies could in fact be designed from much more limited Positivstellens\"atze dating back to the early 20th century that only certify positivity of a polynomial globally. More precisely, we show that any inner approximation to the cone of positive homogeneous polynomials that is arbitrarily tight can be turned into a converging hierarchy {\gh of lower bounds} for {\gh general polynomial minimization problems} with compact feasible sets. This in particular leads to a semidefinite programming-based hierarchy that relies solely on Artin's solution to Hilbert's 17th problem. We also use a classical result of Poly\'a on global positivity of even forms to construct an ``optimization-free'' converging hierarchy for general {\gh polynomial minimization problems} with compact feasible sets. This hierarchy only requires polynomial multiplication and checking nonnegativity of coefficients of certain fixed polynomials. As a corollary, we obtain new linear programming and second-order cone programming-based hierarchies for {\gh polynomial minimization problems} that rely on the recently introduced concepts of dsos (diagonally dominant sum of squares) and sdsos (scaled diagonally dominant sum of squares) polynomials. {\gh We remark that the scope of this paper is theoretical at this stage \aaa{as our hierarchies---though they involve at most two sum of squares constraints or only elementary arithmetic at each level---require the use of bisection and increase the number of variables (resp. degree) of the problem by the number of inequality constraints plus three (resp. by a factor of two).
}} % (resp. the degree) roughly by the number of constraints (resp. by a factor of two).}}
\end{abstract}

% REQUIRED

% REQUIRED

\section{Introduction}

A polynomial optimization problem (POP) is an optimization problem of the form
\begin{equation}\label{eq:POP}
\begin{aligned}
&\inf_{x \in \mathbb{R}^n} &&p(x) \\
&\text{s.t. } &&g_i(x)\geq 0, ~i=1,\ldots,m,
\end{aligned}
\end{equation}
where $p, g_i,~i=1,\ldots,m,$ are polynomial functions in $n$ variables $x\mathrel{\mathop{:}}=(x_1,\ldots,x_n)$ and with real coefficients. It is well-known that polynomial optimization is a hard problem to solve in general. For example, simply testing whether the optimal value of problem (\ref{eq:POP}) is smaller than or equal to some rational number $k$ is NP-hard already when the objective is quadratic and the constraints are linear \cite{pardalos1991quadratic}. Nevertheless, these problems remain topical due to their numerous applications throughout engineering, operations research, and applied mathematics (see, e.g., \cite{lasserre2009moments,blekherman2012semidefinite,ahmadiOR_letters}). In this paper, we are interested in obtaining lower bounds on the optimal value of problem (\ref{eq:POP}). We focus on a class of methods which construct hierarchies of tractable convex optimization problems whose optimal values are lower bounds on the optimal value of (\ref{eq:POP}), with convergence to it as the sequence progresses.  This implies that even though the original POP is nonconvex, one can obtain increasingly accurate lower bounds on its optimal value by solving convex optimization problems. One method for constructing these hierarchies of optimization problems that has gained attention in recent years relies on the use of \emph{Positivstellens\"atze} (see, e.g., \cite{Laurent_survey} for a survey). Positivstellens\"atze are algebraic identities that certify infeasibility of a set of polynomial inequalities, or equivalently\footnote{Note that the set $\{x \in \mathbb{R}^n ~|~ g_1(x) \geq 0, \ldots, g_m(x)\geq 0\}$ is empty if and only if $-g_1(x)>0$ on the set $\{x \in \mathbb{R}^n ~|~ g_2(x) \geq 0, \ldots, g_m(x) \geq 0\}$.}, positivity of a polynomial on a basic semialgebraic set. (Recall that a basic semialgebraic set is a set defined by finitely many polynomial inequalities.) These Positivstellens\"atze can be used to prove {\gh lower bounds} on POPs. Indeed, if we denote the feasible set of (\ref{eq:POP}) by $S$, the optimal value of problem (\ref{eq:POP}) is equivalent to 
\begin{equation} \label{eq:gamma.opt}
\begin{aligned}
&\sup_{\gamma} &&\gamma\\
&\text{s.t. } &&p(x)-\gamma \geq 0,~\forall x \in S.
\end{aligned}
\end{equation}
Hence if $\gamma$ is a strict {\gh lower bound} on (\ref{eq:POP}), we have that $p(x)-\gamma>0$ on $S$, a fact that can be certified using Positivstellens\"atze.  At a conceptual level, hierarchies that provide lower bounds on (\ref{eq:POP}) are constructed thus: we fix the ``size of the certificate'' at each level of the hierarchy and search for the largest $\gamma$ such that the Positivstellens\"atze at hand can certify positivity of $p(x)-\gamma$ over $S$ with a certificate of this size. As the sequence progresses, we increase the size of the certificates allowed, hence obtaining increasingly accurate lower bounds on (\ref{eq:POP}). 

Below, we present three of the better-known Positivstellens\"atze, given respectively by Stengle~\cite{stengle1974nullstellensatz}, Schm\"udgen~\cite{schmudgen1991thek}, and Putinar~\cite{putinar1993positive}. These all rely on {sum of squares} certificates. We recall that a polynomial is a \emph{sum of squares} (sos) if it can be written as a sum of squares of other polynomials. We start with Stengle's Positivstellensatz, which certifies infeasibility of a set of polynomial inequalities. It is sometimes referred to as ``the Positivstellensatz'' in related literature as it requires no assumptions, contrarily to Schm\"udgen and Putinar's theorems which can be viewed as refinements of Stengle's result under additional assumptions. {\gh This Positivstellensatz was in fact discovered by Krivine in 1964 \cite{krivine}, and rediscovered by Stengle later\footnote{We thank an anonymous referee for pointing this out to us and for providing us with the appropriate references.}; see \cite[Section 4.7]{prestel2013positive} for a more complete history of this result.} 

\begin{theorem}[Stengle's Positivstellensatz~\cite{stengle1974nullstellensatz}]\label{th:stengle}
The {\gh closed} basic semialgebraic set $$S=\{x\in \mathbb{R}^n ~|~ g_1(x)\geq 0,\ldots, g_m(x)\geq 0\}$$ is empty if and only if there exist sum of squares polynomials $s_0(x)$,$s_1(x)$,$\ldots$, $s_m(x)$, $s_{12}(x)$, $s_{13}(x)$,$\ldots$, $s_{123\ldots m}(x)$ such that
$$-1=s_0(x)+\sum_{i} s_{i}(x)g_{i}(x) +\sum_{\{i,j\}} s_{ij}(x)g_{i}(x)g_{j}(x)+\ldots+s_{123\ldots m}(x)g_{1}(x)\ldots g_{m}(x).$$
\end{theorem}
The next two theorems, due to Schm\"udgen and Putinar, certify positivity of a polynomial $p$ over a {\gh closed} basic semialgebraic set $S$. They impose additional compactness assumptions comparatively to Stengle's Positivstellensatz.
\begin{theorem}[Schm\"udgen's Positivstellensatz \cite{schmudgen1991thek}]\label{th:schmudgen}
	Assume that the set $$S=\{x \in \mathbb{R}^n ~|~ g_1(x)\geq 0, \ldots, g_m(x)\geq 0\}$$ is compact. If a polynomial $p$ is positive on $S$, then $$p(x)=s_0(x)+\sum_{i} s_{i}(x)g_{i}(x) +\sum_{\{i,j\}} s_{ij}(x)g_{i}(x)g_{j}(x)+\ldots+s_{123\ldots m}(x)g_{1}(x)\ldots g_{m}(x),$$ where $s_0(x)$,$s_1(x)$,$\ldots$, $s_m(x)$, $s_{12}(x)$, $s_{13}(x)$,$\ldots$, $s_{123\ldots m}(x)$ are sums of squares.
	\end{theorem}

\begin{theorem}[Putinar's Positivstellensatz~\cite{putinar1993positive}]\label{th:putinar}
	Let $$S=\{x \in \mathbb{R}^n ~|~ g_1(x)\geq 0, \ldots, g_m(x)\geq 0\}$$ and assume that $\{g_1,\ldots,g_m\}$ satisfy the Archimedean property, i.e., there exists $N \in~\mathbb{N}$ such that $$N-\sum_i x_i^2=\sigma_0(x)+\sigma_1(x) g_1(x)+\ldots+\sigma_m(x) g_m(x),$$ where $\sigma_1(x),\ldots,\sigma_m(x)$ are sums of squares. If a polynomial $p$ is positive on $S$, then $$p(x)=s_0(x)+s_1(x)g_1(x)+\ldots+s_m(x)g_m(x),$$ where $s_1(x),\ldots,s_m(x)$ are sums of squares.
\end{theorem}
Note that these three Positivstellens\"atze involve in their expressions sum of squares polynomials of unspecified degree. To construct hierarchies of tractable optimization problems for (\ref{eq:gamma.opt}), we fix this degree: at level $r$, we search for the largest $\gamma$ such that positivity of $p(x)-\gamma$ over $S$ can be certified using the Positivstellens\"atze where the degrees of all sos polynomials are taken to be less than or equal to $2r$. Solving each level of these hierarchies is then a semidefinite program (SDP). This is a consequence of the fact that one can optimize over (or test membership to) the set of sum of squares polynomials of fixed degree using semidefinite programming \cite{sdprelax,PabloPhD,lasserre_moment}. Indeed, a polynomial $p$ of degree $2d$ and in $n$ variables is a sum of squares if and only if there exists a symmetric matrix $Q\succeq 0$ such that $p(x)=z(x)^TQz(x)$, where $z(x)=(1,x_1,\ldots,x_n,\ldots,x_n^d)^T$ is the standard vector of monomials in $n$ variables and of degree less than or equal to $d$. We remark that the hierarchy obtained from Stengle's Positivstellensatz was proposed and analyzed by Parrilo in \cite{sdprelax}; the hierarchy obtained from Putinar's Positivstellensatz was proposed and analyzed by Lasserre in~\cite{lasserre_moment}. There have been more recent works that provide constructive proofs of Schm\"udgen and Putinar's Positivstellens\"atze; see \cite{averkov, schweighofer2002algorithmic, schweighofer2005optimization}. These proofs rely on other Positivstellens\"atze, e.g., a result by Poly\'a (see Theorem \ref{th:polya} below) in \cite{schweighofer2002algorithmic, schweighofer2005optimization}, and the same result by Poly\'a, Farkas' lemma, and Stengle's Positivstellensatz in \cite{averkov}. {\gh We would like to thank an anonymous referee for pointing out that the construction in \cite{schweighofer2002algorithmic} can be used to develop converging hierarchies of lower bounds for POPs with compact feasible sets. These hierarchies rely on Gr\"obner bases computations and linear programs involving only two variables. Some experiments with this technique were carried out by Datta~\cite{datta} \aaa{and Averkov has more recently shown~\cite{averkov} that the (potentially expensive) Gr\"obner bases computations can be avoided in this approach}. Other recent research efforts relating to Positivstellens\"atze have been focused around deriving complexity bounds for Schm\"udgen and Putinar's Positivstellens\"atze; see \cite{nie2007Putinarcomplexity,schweighofer2004complexity}.}

%There has further been an effort to derive complexity bounds for Schm\"udgen and Putinar's Positivstellens\"atze in recent years; see \cite{nie2007complexity,schweighofer2004complexity}.

%We briefly discuss the assumptions required for both Positivstellens\"atze to hold. Theorem \ref{th:schmudgen} requires only compactness whereas Theorem \ref{th:putinar} requires a stronger assumption, namely an algebraic proof of compactness (the so-called Archimedean property). The assumption that we require in the hierarchies that we derive also relates to compactness: we need to know the value of a constant $R$ such that $S \subseteq B(0,R)$, where $B(0,R)$ is the ball of radius $R$ and centered at $0.$
On a historical note, Stengle, Schm\"udgen, and Putinar's Positivstellens\"atze were derived in the latter half of the 20th century. As mentioned previously, they all certify positivity of a polynomial over an arbitrary basic semialgebraic set (modulo compactness assumptions). By contrast, there are Positivstellens\"atze from the early 20th century that certify positivity of a polynomial \emph{globally.} Perhaps the most well-known Positivstellensatz of this type is due to Artin in 1927, in response to Hilbert's 17th problem. Artin shows that any nonnegative polynomial is a sum of squares of rational functions. Here is an equivalent formulation of this statement:

\begin{theorem}[Artin \cite{artin1927}]\label{th:artin}
For any nonnegative polynomial $p$, there exists {\gh a nonzero sos polynomial $q$} such that $p\cdot q$ is a sum of squares.
\end{theorem}
To the best of our knowledge, in this area, all converging hierarchies of lower bounds for POPs are based off of Positivstellens\"atze that certify nonnegativity of a polynomial over an arbitrary basic semialgebraic set. In this paper, we show that in fact, under compactness assumptions, it suffices to have only global certificates of nonnegativity (such as the one given by Artin) to produce a converging hierarchy for general POPs. As a matter of fact, even weaker statements that apply only to globally positive (as opposed to globally nonnegative) forms are enough to derive converging hierarchies for POPs. Examples of such statements are due to Habicht \cite{habicht1939zerlegung} and Reznick \cite{Reznick_Unif_denominator}. With such an additional positivity assumption, more can usually be said about the structure of the polynomial $q$ in Artin's result. Below, we present the result by Reznick.

\begin{theorem}[Reznick \cite{Reznick_Unif_denominator}]\label{th:reznick.uniform}
For any positive definite form $p$, there exists $r \in \mathbb{N}$ such that $p(x) \cdot (\sum_i x_i^2)^r$ is a sum of squares.	
\end{theorem}
We show in this paper that this Positivstellensatz also gives rise to a converging hierarchy for POPs with a compact feasible set similarly to the one generated by Artin's Positivstellensatz.

Through their connections to sums of squares, the two hierarchies obtained using the theorems of Reznick and Artin are semidefinite programming-based. In this paper, we also derive an ``optimization-free'' converging hierarchy for POPs with compact feasible sets where each level of the hierarchy only requires that we be able to test nonnegativity of the coefficients of a given fixed polynomial. To the best of our knowledge, this is the first converging hierarchy of lower bounds for POPs which does not require that convex optimization problems be solved at each of its levels. To construct this hierarchy, we use a result of Poly\'a \cite{Polya}, which just like Artin's and Reznick's Positivstellens\"atze, certifies global positivity of forms. However this result is restricted to even forms. Recall that a form $p$ is \emph{even} if each of the variables featuring in its individual monomials has an even power. This is equivalent (see \cite[Lemma 2]{de2005equivalence}) to $p$ being invariant under change of sign of each of its coordinates, i.e., \begin{equation*}
p(x_1,\ldots,x_n)=p(-x_1,\ldots,x_n)=\cdots=p(x_1,\ldots,-x_n).
\end{equation*}
\begin{theorem}[Poly\'a \cite{Polya}]\label{th:polya}
For any positive definite even form $p$, there exists $r \in \mathbb{N}$ such that $p(x) \cdot (\sum_i x_i^2)^r$ has nonnegative coefficients.
%\footnote{A perhaps better-known but equivalent formulation of this theorem is the following: for any form $h$ that is positive on the standard simplex, there exists $r \in \mathbb{N}$ such that $h(x)\cdot (\sum_i x_i)^r$ has nonnegative coefficients. The two formulations are equivalent by simply letting $p(x)=h(x^2)$.} 
\end{theorem}

{\gh A perhaps better-known but equivalent formulation of this theorem is the following: for any form $h$ that is positive on the standard simplex, there exists $r \in \mathbb{N}$ such that $h(x)\cdot (\sum_i x_i)^r$ has nonnegative coefficients. The two formulations are equivalent by simply letting $p(x)=h(x^2)$. The latter formulation has been used to derive \aaa{similar optimization-free} converging hierarchies of lower bounds for polynomial minimization problems over the simplex; see, e.g., \cite{de2006ptas,deKlerk_StableSet_copositive}.}
	%The former formulation has been used to derive hierarchies of bounds on copositive programs, i.e., conic programs over the cone of copositive matrices \cite{PabloPhD}. This is due to the close connection between copositivity of matrices and nonnegativity of even forms. Indeed, a matrix $M$ is copositive if and only if the associated even form $(x.^2)^T Mx.^2$ is globally nonnegative, where $x.^2=(x_1^2,\ldots,x_n^2)^T$.}

%Our aforementioned hierarchy 
Our aforementioned optimization-free hierarchy \aaa{also} enables us to obtain linear programming (LP) and second-order cone programming (SOCP)-based hierarchies for general POPs with compact feasible sets that rely on the concepts of \emph{dsos} and \emph{sdsos} polynomials. These are recently introduced inner approximations to the set of sos polynomials that have shown much better scalability properties in practice~\cite{iSOS_journal}.

As a final remark, we wish to stress the point that the goal of this paper is first and foremost theoretical, i.e., to provide methods for constructing converging hierarchies of lower bounds for POPs using as sole building blocks certificates of global positivity. {\gh We do not make any claims that these hierarchies can outperform the popular existing hierarchies due, e.g., to Lasserre \cite{lasserre_moment} and Parrilo \cite{sdprelax}. In particular, all hierarchies that we generate \aaa{increase the number of variables and the degree of the polynomials involved from $n$ to $n+m+3$, and from $d$ to $2d,$ respectively. They also} necessitate the use of bisection, which, while not a problem in theory, increases the computational overload.} \aaa{We remark however that each level of our hierarchies only involves either one sum of squares constraint (the hierarchy based on the certificate of Reznick; Theorem \ref{th:reznick.hierarchy}), two sum of squares constraints (the hierarchy based on the certificate of Artin; Theorem \ref{th:artin.hierarchy}), or nothing but elementary computations (the hierarchy based on the certificate of Poly\'a; Theorem \ref{th:sms.hierarchy}). By contrast, each level of the hierarchy based on Putinar's (resp. Schm\"udgen's) certificate involves $m+1$ (resp. $2^m$) sum of squares constraints, but necessitates no need to use bisection or to increase the number of variables/degree of the problem. Similarly, a hierarchy based on Stengle's certificate, which would work by showing infeasibility of the constraints $\{\gamma-p(x)\geq 0, g_1(x)\geq 0,\ldots, g_m(x)\geq 0 \}$, requires the use of bisection on $\gamma$ and $2^{m+1}$ sum of squares constraints in each level, but necessitates no increase in the number of variables/degree of the problem. Of course, such comparisons would become more meaningful if one could also relate the quality of the bounds obtained from the different approaches. Some remarks on why it is nontrivial to connect our hierarchies to previous ones in this sense are made in Section~\ref{sec:conclusion}.}

%\aaa{On the other hand, each level of our hierarchy requires only one sum of squares constraint (Theorem \ref{th:reznick.hierarchy}), or two sum of squares constraints (Theorem \ref{th:artin.hierarchy}), or nothing but elementary computations (Theorem \ref{th:sms.hierarchy}). In comparison, the hierarchy based on Putinar's certificate, for example, needs $m+1$ sum of squares constraints in each level, but does not involve bisection or an increase in the number of variables/degree of the problem. Of course, such comparisons would become more meaningful if one could also relate the quality of the bounds obtained from the different approaches. Some remarks on why statements of this type are nontrivial to obtain are made in Section~\ref{sec:conclusion}.}

%We do believe however that the optimization-free hierarchy presented in Section \ref{subsec:opt.free} could potentially be of interest in large-scale applications where the convex optimization problems appearing in traditional hierarchies are too cumbersome to solve.

\subsection{Outline of the paper} The paper is structured as follows. In Section \ref{sec:nonneg.approx}, we show that if one can inner approximate the cone of positive definite forms arbitrarily well (with certain basic properties), then one can produce a converging hierarchy of lower bounds for POPs with compact feasible sets (Theorem~\ref{th:hierarchy}). This relies on a reduction (Theorem~\ref{th:slb}) that reduces the problem of certifying a strict lower bound on a POP to that of proving positivity of a certain form. In Section \ref{sec:sdp.hierarchy}, we see how this result can be used to derive semidefinite programming-based converging hierarchies (Theorems \ref{th:reznick.hierarchy} and \ref{th:artin.hierarchy}) from the Positivstellens\"atze by Artin (Theorem \ref{th:artin}) and Reznick (Theorem \ref{th:reznick.uniform}). In Section~\ref{sec:opt.free.LP.SOCP}, we derive an optimization-free hierarchy (Theorem \ref{th:sms.hierarchy}) from the Positivstellensatz of Poly\'a (Theorem \ref{th:polya}) as well as LP and SOCP-based hierarchies which rely on dsos/sdsos polynomials (Corollary \ref{cor:dsos.sdsos.hierarchy}). We conclude with a few open problems in Section \ref{sec:conclusion}.

\subsection{Notation and basic definitions}

We use the standard notation $A\succeq 0$ to denote that a symmetric matrix $A$ is positive semidefinite. Recall that a \emph{form} is a homogeneous polynomial, i.e., a polynomial whose monomials all have the same degree. We denote the degree of a form $f$ by $deg(f)$. We say that a form $f$ is \emph{nonnegative} (or positive semidefinite) if $f(x) \geq 0$, for all $x \in \mathbb{R}^n$ (we write $f\geq 0$). A form $f$ is \emph{positive definite} (pd) if $f(x) >0,$ for all nonzero $x$ in $\mathbb{R}^n$ (we write $f>0$). Throughout the paper, we denote the set of forms (resp. the set of nonnegative forms) in $n$ variables and of degree $d$ by $H_{n,d}$ (resp $P_{n,d}$). We denote the ball of radius $R$ and centered at the origin by $B(0,R)$ and the unit sphere in $x$-space, i.e., $\{x\in \mathbb{R}^n~|~||x||_2=1\}$, by $S_x$. We use the shorthand $f(y^2-z^2)$ for $y,z \in \mathbb{R}^n$ to denote $f(y_1^2-z_1^2,\ldots,y_n^2-z_n^2).$ We say that a scalar $\gamma$ is a strict lower bound on (\ref{eq:POP}) if $p(x)>\gamma,~\forall x\in S$. Finally, we ask the reader to carefully read Remark~\ref{rem:notation} which contains the details of a notational overwriting occurring before Theorem \ref{th:hierarchy} and valid from then on throughout the paper. This overwriting makes the paper much simpler to parse.

\section{Constructing converging hierarchies for POP using global certificates of positivity}\label{sec:nonneg.approx}

Consider the polynomial optimization problem in (\ref{eq:POP}) and denote its optimal value by $p^*$. Let $d$ be such that $2d$ is the smallest even integer larger than or equal to the maximum degree of $p, g_i, i=1,\ldots,m$. We denote the feasible set of our optimization problem by $$S=\{x \in \mathbb{R}^n~|~g_i(x)\geq 0, i=1,\ldots,m\}$$ and assume that $S$ is contained within a ball of radius $R$. From this, it is easy to provide (possibly very loose) upper bounds on $g_i(x)$ over the set $S$: as $S$ is contained in a ball of radius $R$, we have $|x_i|\leq R$, for all $i=1,\ldots,n$. We then use this to upper bound each monomial in $g_i$ and consequently $g_i$ itself. We use the notation $\eta_i$ to denote these upper bounds, i.e., $g_i(x) \leq \eta_i$, for all $i=1,\ldots,m$ and for all $x\in S$. Similarly, we can provide an upperbound on $-p(x)$. We denote such a bound by $\beta$, i.e., $-p(x) \leq \beta,$ $\forall x \in S.$

The goal of this section is to produce a method for constructing converging hierarchies of lower bounds for POPs if we have access to arbitrarily accurate inner approximations of the set of positive definite forms. The first theorem (Theorem \ref{th:slb}) connects lower bounds on (\ref{eq:POP}) to positive definiteness of a related form. The second theorem (Theorem~\ref{th:hierarchy}) shows how this can be used to derive a hierarchy for POPs.

\begin{theorem}\label{th:slb} Consider the general polynomial optimization problem in (\ref{eq:POP}) and recall that $d$ is such that $2d$ is the smallest even integer larger than or equal to the maximum degree of $p, g_i, i=1,\ldots,m$. Suppose $S \subseteq B(0,R)$ for some positive scalar $R$. Let $\eta_i, i=1,\ldots,m$ (resp.  $\beta$) be any finite upper bounds on $g_i(x), i=1,\ldots,m$ (resp. $-p(x)$).
	
	Then, a scalar $\gamma$ is a strict lower bound on (\ref{eq:POP}) if and only if the homogeneous sum of squares polynomial
	\begin{align}
	f_\gamma(x,s,y)\mathrel{\mathop{:}}=&\left( \gamma y^{2d}-y^{2d}p(x/y)-s_0^2y^{2d-2} \right)^2+\sum_{i=1}^m\left(y^{2d}g_i(x/y)-s_i^2 y^{2d-2}\right)^2 \label{eq:f.gamma}\\
	&+\left((R+\sum_{i=1}^m \eta_i +\beta+\gamma)^dy^{2d}-(\sum_{i=1}^n x_i^2+\sum_{i=0}^m s_i^2)^d-s_{m+1}^{2d}\right)^2 \nonumber
	\end{align} of degree $4d$ and in $n+m+3$ variables $(x_1,\ldots,x_n,s_0,\ldots,s_m,s_{m+1},y)$ is {\gh positive definite.\footnote{The reader will observe in the proof that the variables $(s_0,\ldots,s_{m+1})$ will serve as slack variables and the variable $y$ will be used for homogenization.}}
\end{theorem}

\begin{proof}\renewcommand{\qedsymbol}{}
It is easy to see that $\gamma$ is a strict lower bound on (\ref{eq:POP}) if and only if the set $$T\mathrel{\mathop{:}}=\{x \in \mathbb{R}^n ~|~ \gamma- p(x) \geq 0; \quad g_i(x)\geq 0, i=1,\ldots,m;\quad \sum_i x_i^2 \leq R\}$$ is empty. Indeed, if $T$ is nonempty, then there exists a point $x \in S$ such that $p(x)\leq \gamma$. This implies that $\gamma$ cannot be a strict lower bound on (\ref{eq:POP}). Conversely, if $T$ is empty, the intersection of $S$ with $\{x~|~\gamma-p(x)\geq 0\}$ is empty, which implies that $\forall x \in S$, $p(x)>\gamma$.

We now define the set:
\begin{equation}\label{eq:def.Ts}
\begin{aligned}
T_s=\{(x,s) \in \mathbb{R}^{n+m+2} ~|~ \gamma-p(x)=s_0^2; \quad g_i(x)=s_i^2, i=1,\ldots,m; \phantom\}\\
 \quad \phantom\{(R+\sum_{i=1}^m \eta_i+\beta+\gamma)^d -(\sum_{i=1}^n x_i^2+\sum_{i=0}^m s_i^2)^d-s_{m+1}^{2d}=0\}. 
\end{aligned}
\end{equation}
Note that $T_s$ is empty if and only if $T$ is empty. Indeed, if $T_s$ is nonempty, then there exists $x \in \mathbb{R}^n$ and $s \in \mathbb{R}^{m+2}$ such that the three sets of equations are satisfied. This obviously implies that $\gamma -p(x)\geq 0$ and that $g_i(x)\geq 0$, for all $i=1,\ldots,m.$ It further implies that $\sum_i x_i^2 \leq R$ as by assumption, if $x \in S$, then $x$ is in a ball of radius $R$. 
Conversely, suppose now that $T$ is nonempty. There exists $x$ such that $\gamma-p(x)\geq 0$, $g_i(x)\geq 0$ for $i=1,\ldots,m$, and $\sum_i x_i^2 \leq R.$ Hence, there exist $s_0,\ldots,s_m$ such that $$\gamma-p(x)=s_0^2 \text{ and } g_i(x)={s_i}^2,~ i=1,\ldots,m.$$ Combining the fact that $\sum_i {x_i}^2 \leq R$ and the fact that $\eta_i$, $i=1,\ldots,m$ (resp. $\gamma+\beta$) are upperbounds on $g_i$ (resp. $\gamma-p({x})$), we obtain: 
$$R+\sum_{i=1}^m \eta_i+\beta+\gamma \geq \sum_{i=1}^n {x_i}^2+\sum_{i=0}^m {s_i}^2.$$  By raising both sides of the inequality to the power $d$, we show the existence of ${s}_{m+1}$.

We now show that $T_s$ is empty if and only if $f_{\gamma}(x,s,y)$ is positive definite. Suppose that $T_s$ is nonempty, i.e., there exists $({x},{s}) \in \mathbb{R}^{n+m+2}$ such that the equalities given in (\ref{eq:def.Ts}) hold. Note then that $f_{\gamma}(x,s,1)=0$. As $(x,s,1)$ is nonzero, this implies that $f_{\gamma}(x,s,y)$ is not positive definite. 

For the converse, assume that $f_{\gamma}(x,s,y)$ is not positive definite. As $f_{\gamma}(x,s,y)$ is a sum of squares and hence nonnegative, this means that there exists nonzero $(x,s,y)$ such that $f(x,s,y)=0$. We proceed in two cases. If $y \neq 0$, it is easy to see that $(x/y,s/y) \in T_s$ and $T_s$ is nonempty. Consider now the case where $y=0$. The third square in $f_{\gamma}$ being equal to zero gives us: $$-(\sum_i x_i^2+\sum_{i=0}^m s_i^2)^d=s_{m+1}^{2d}.$$
This implies that $s_{m+1}=0$ and that $x_1=\ldots=x_m={s_0}=\ldots={s}_m=0$ which contradicts the fact that $(x,s,y)$ is nonzero.
\end{proof}

\begin{remark}
Note that Theorem \ref{th:slb} implies that testing feasibility of a set of polynomial inequalities is no harder than checking whether a homogeneous polynomial that is sos has a zero. Indeed, as mentioned before, the basic semialgebraic set $$\{x~|~g_1(x)\geq 0,\ldots, g_m(x)\geq 0\}$$ is empty if and only if $\gamma=0$ is a strict lower bound on the POP
\begin{equation*}
\begin{aligned}
&\inf_x &&-g_1(x)\\
&\text{s.t. } &&g_2(x) \geq 0,\ldots,g_{m}(x)\geq 0.
\end{aligned}
\end{equation*}
In principle, this reduction can open up new possibilities for algorithms for testing feasibility of a basic semialgebraic set. For example, the work in \cite{AAA_Cubic_vec_field} shows that positive definiteness of a form $f$ is equivalent to global asymptotic stability of the polynomial vector field $\dot{x}=-\nabla f(x).$ One could as a consequence search for Lyapunov functions, as is done in \cite[Example 2.1.]{AAA_Cubic_vec_field}, to certify positivity of forms. Conversely, simulating trajectories of the above vector field can be used to minimize $f$ and potentially find its nontrivial zeros, which, by our reduction, can be turned into a point that belongs to the basic semialgebraic set at hand. 

We further remark that one can always take the degree of the sos form $f_{\gamma}$ in (\ref{eq:f.gamma}) whose positivity is under consideration to be equal to four. This can be done by changing the general POP in (\ref{eq:POP}) to only have quadratic constraints and a quadratic objective via an iterative introduction of new variables and new constraints in the following fashion: $x_{ij}=x_ix_j$.

\end{remark}

\begin{remark}[Notational remark]\label{rem:notation}
	As a consequence of Theorem~\ref{th:slb}, we now know that certifying lower bounds on (\ref{eq:POP}) is equivalent to proving positivity of the form $f_{\gamma}$ that appears in (\ref{eq:f.gamma}). {\gh To simplify notation, we define
	 $$N\mathrel{\mathop{:}}=n+m+3 \text{ and } D=2d,$$
	 where $n$ is the dimension of the decision variable of problem (\ref{eq:POP}), $d$ is such that $2d$ is the smallest even integer larger than or equal to the maximum degree of $g_i$ and $p$ in (\ref{eq:POP}), and $m$ is the number of constraints of problem (\ref{eq:POP}).
	 Note now that the form $f_{\gamma}$ is a polynomial in $N$ variables and of degree $2D$.
}

%	 To simplify notation, we take this form to have $n$ variables and be of degree $2d$ from now on (except for our Positivstellens\"atze in Corollaries \ref{cor:SDP.psatz} and \ref{cor:Positiv.LP} which stand on their own). To connect back to problem (\ref{eq:POP}) and the original notation, the reader should replace every occurrence of $n$ and $d$ in the future as follows: $$n \leftarrow n+m+3, \quad d \leftarrow 2d.$$ Recall that $n$ was previously the dimension of the decision variable of problem (\ref{eq:POP}), $d$ was such that $2d$ is the smallest even integer larger than or equal to the maximum degree of $g_i$ and $p$ in (\ref{eq:POP}), and $m$ was the number of constraints of problem (\ref{eq:POP}).
	\end{remark}

Our next theorem shows that, modulo some technical assumptions, if one can inner approximate the set of positive definite forms arbitrarily well (conditions (a) and (b)), then one can construct a converging hierarchy for POPs.

\begin{theorem} \label{th:hierarchy}
Let $K_{n,2d}^r$ be a sequence of sets (indexed by $r$) of homogeneous polynomials in $n$ variables and of degree $2d$ with the following properties:
\begin{enumerate}[(a)]
	\item $K_{n,2d}^r \subseteq P_{n,2d}, \forall r,$ and there exists a pd form $s_{n,2d} \in K_{n,2d}^0.$
	\item If $p>0$, then $\exists r \in \mathbb{N}$ such that $p \in K_{n,2d}^r.$
	\item $K_{n,2d}^r \subseteq K_{n,2d}^{r+1}$, $\forall r$.
	\item If $p \in K_{n,2d}^r$, then $\forall \epsilon \in [0,1]$, $p+\epsilon s_{n,d} \in K_{n,2d}^{r}.$ 
	
\end{enumerate}

Recall the definition of $f_{\gamma}(z)$ given in (\ref{eq:f.gamma}). Consider the hierarchy of optimization problems indexed by $r$:
\begin{equation}\label{eq:hierarchy}
\begin{aligned}
l_r\mathrel{\mathop{:}}=&\sup_{\gamma} &&\gamma\\
&\text{s.t. } &&f_{\gamma}(z)-\frac{1}{r}s_{N,2D}(z) \in K_{N,2D}^r.
\end{aligned}
\end{equation}
Then, $l_r \leq p^*$ for all $r$, $\{l_r\}$ is nondecreasing, and $\lim_{r \rightarrow \infty} l_r=p^*.$

\end{theorem}

\begin{proof} 
	
	We first show that the sequence $\{l_r\}$ is upperbounded by $p^*$. Suppose that a scalar $\gamma$ satisfies $$f_{\gamma}(z) -\frac{1}{r}s_{N,2D}(z) \in K_{N,2D}^r.$$ We then have $f_{\gamma}(z)-\frac{1}{r}s_{N,2D}(z) \in P_{N,2D}$ using (a). This implies that $f_{\gamma}(z)~\geq~\frac{1}{r}s_{N,2D}(z) $, and hence $f_{\gamma}$ is pd as $s_{N,2D}$ is pd. {\gh From Theorem \ref{th:slb}, it follows that $\gamma$ has to be a strict lower bound on the optimal value of (\ref{eq:POP}).} As $\gamma < p^*$, we have that $l_r \leq p^*$ for all $r$.\\
	
We now show monotonicity of the sequence $\{l_r\}$. Let $\gamma$ be such that $$f_{\gamma}(z)-\frac{1}{r} s_{N,2D}(z) \in K_{N,2D}^r.$$ We have the following identity:$$f_{\gamma}(z)-\frac{1}{r+1} s_{N,2D}(z)=f_{\gamma}(z)-\frac{1}{r}s_{N,2D}(z)+\frac{1}{r(r+1)}s_{N,2D}(z).$$
Now, using the assumption and properties (c) and (d), we conclude that 
$$f_{\gamma}(z)-\frac{1}{r+1}s_{N,2D}(z) \in K_{N,2D}^{r+1}.$$ This implies that $$\{\gamma ~|~ f_{\gamma}(z)-\frac{1}{r}s_{N,2D}(z) \in K_{N,2D}^r\} \subseteq \{\gamma ~|~ f_{\gamma}(z)-\frac{1}{r+1}s_{N,2D}(z) \in K_{N,2D}^{r+1}\}$$ and that $l_r \leq l_{r+1}.$\\

Note that as the sequence $\{l_r\}$ is upper bounded and nondecreasing, it converges. Let us show that the limit of this sequence is $p^*$. To do this, we show that for any strict lower bound $\gamma$ on (\ref{eq:POP}), there exists a positive integer $r$ such that $f_{\gamma}(z)-\frac{1}{r} s_{N,2D}(z) \in K_{N,2D}^r$. By Theorem~\ref{th:slb}, as $\gamma$ is a strict lower bound, $f_{\gamma}(z)$ is positive definite. {\gh As a form is positive definite if and only if it is positive on the unit sphere, by continuity, there exists a positive integer $r'$ such that $f_{\gamma}(z)-\frac{1}{r'}s_{N,2D}(z)$ is positive definite.} Using (b), this implies that there exists a positive integer $r''$ such that 
\begin{equation}\label{eq:f.gam.in.cone}
f_{\gamma}(z)-\frac{1}{r'}s_{N,2D}(z) \in K_{N,2D}^{r''}.
\end{equation}
We now proceed in two cases. If $r'' \leq r'$, we take $r=r'$ and use property (c) to conclude. If $r' \leq r''$, we have
  $$f_{\gamma}(z)-\frac{1}{r''}s_{N,2D}(z)=f_{\gamma}(z)-\frac{1}{r'}s_{N,2D}(z)+\frac{r''-r'}{r'\cdot r''}s_{N,2D}(z).$$
  We take $r=r''$ and use (\ref{eq:f.gam.in.cone}) and properties (c) and (d) to conclude.
\end{proof}

\begin{remark} Note that condition (d) is subsumed by the more natural condition that $K_{n,d}^r$ be a convex cone for any $n,d,$ and $r$. However, there are interesting and relevant cones which we cannot prove to be convex though they trivially satisfy condition (d) (see Theorem \ref{th:reznick.hierarchy} for an example).
	\end{remark}

\section{Semidefinite programming-based hierarchies obtained from Artin's and Reznick's Positivstellens\"atze}\label{sec:sdp.hierarchy}

In this section, we construct two different semidefinite programming-based hierarchies for POPs using Positivstellens\"atze derived by Artin (Theorem \ref{th:artin}) and Reznick (Theorem \ref{th:reznick.uniform}). To do this, we introduce two sets of cones that we call the Artin and Reznick cones.

\begin{definition}\label{def:reznick.artin.cones}
We define the \emph{Reznick cone} of level $r$ to be $$R_{n,2d}^{r}\mathrel{\mathop{:}}=\{p \in H_{n,2d}~|~ p(x)\cdot \left(\sum_{i=1}^n x_i^2\right)^r \text{ is sos}\}.$$
Similarly, we define the \emph{Artin cone} of level $r$ to be $$A_{n,2d}^{r} \mathrel{\mathop{:}}=\{p \in H_{n,2d}~|~ p(x) \cdot q(x) \text{ is sos for some {\gh nonzero} sos form $q$ of degree $2r$} \}.$$
	
\end{definition}

We show that both of these cones produce hierarchies of the type discussed in Theorem \ref{th:hierarchy}. Recall that $p^*$ is the optimal value of problem (\ref{eq:POP}) and that {\gh $f_{\gamma}$ is a polynomial in $N$ variables and of degree $2D$ as defined in (\ref{eq:f.gamma}) and Remark \ref{rem:notation}.}

\begin{theorem}\label{th:reznick.hierarchy}
	Consider the hierarchy of optimization problems indexed by $r$:
	\begin{equation}\label{eq:reznick.hierarchy}
	\begin{aligned}
	l_r\mathrel{\mathop{:}}=&\sup_{\gamma} &&\gamma\\
	&\text{s.t. } &&f_{\gamma}(z)-\frac{1}{r}(\sum_{i=1}^N z_i^2)^{D} \in R_{N,2D}^r.
	\end{aligned}
	\end{equation}
	Then, $l_r \leq p^*$ for all $r$, $\{l_r\}$ is nondecreasing, and $\lim_{r \rightarrow \infty} l_r=p^*.$
\end{theorem}

\begin{proof}
It suffices to show that the Reznick cones $R_{n,2d}^r$ satisfy properties (a)-(d) in Theorem \ref{th:hierarchy}. The result will then follow from that theorem. For property (a), it is clear that, as $(\sum_i x_i^2)^r>0$ and $p(x) \cdot (\sum_i x_i^2)^r$ is a sum of squares and hence nonnegative, $p(x)$ must be nonnegative, so $R_{n,2d}^r \subseteq P_{n,2d}.$ Furthermore, the form $s_{n,2d}\mathrel{\mathop{:}}=(\sum_i x_i^2)^{d}$ belongs to $R_{n,2d}^0$ and is positive definite. Property (b) is verified as a consequence of Theorem \ref{th:reznick.uniform}. For (c), note that if $p(x) \cdot (\sum_i x_i^2)^r$ is sos, then $p(x) \cdot (\sum_i x_i^2)^{r+1}$ is sos since the product of two sos polynomials is sos. Finally, for property (d), note that $R_{n,2d}^r$ is a convex cone. Indeed, for any $\lambda \in [0,1]$, $$(\lambda p(x) +(1-\lambda) q(x)) \cdot (\sum_i x_i^2)^r=\lambda p(x) (\sum_i x_i^2)^r+(1-\lambda) q(x) (\sum_i x_i^2)^r $$
	is sos if $p$ and $q$ are in $R_{n,2d}^r$. Combining the fact that $R_{n,2d}^r$ is a convex cone and the fact that $(\sum_i x_i^2)^d \in R_{n,d}^r$, we obtain (d).
	\end{proof}

\begin{remark}\label{rem:bisection}
	To solve a fixed level $r$ of the hierarchy given in Theorem \ref{th:reznick.hierarchy}, one must proceed by \emph{bisection} on $\gamma$ \aaa{(since the parameter $\gamma$ appears with a power in the definition of $f_\gamma$ in (\ref{eq:f.gamma})).} Bisection here would produce a sequence of upper bounds $\{U_k\}$ and lower bounds $\{L_k\}$ on $l_r$ as follows. At iteration $k$, we test whether $\gamma=\frac{U_k+L_k}{2}$ is feasible for (\ref{eq:reznick.hierarchy}). If it is, then we take $L_{k+1}=\frac{U_k+L_k}{2}$ and $U_{k+1}=U_k$. If it is not, we take $U_{k+1}=\frac{U_k+L_k}{2}$ and $L_{k+1}=L_k$. We stop when $|U_{k_\epsilon}-L_{k_\epsilon}|<\epsilon$, where $\epsilon$ is a prescribed accuracy, and the algorithm returns $l_{r,\epsilon}=L_{k_\epsilon}.$ Note that $l_{r}-\epsilon \leq l_{r,\epsilon} \leq l_r$ and that to obtain $l_{r,\epsilon}$, one needs to take a logarithmic (in $\frac{1}{\epsilon}$) number of steps using this method. 
	
	Hence, solving the $r^{th}$ level of this hierarchy using bisection can be done by semidefinite programming. Indeed, for a fixed $r$ and $\gamma$ given by the bisection algorithm, one simply needs to test membership of $$\left( f_{\gamma}(z)-\frac{1}{r}(\sum_i z_i^2)^D\right) \cdot (\sum_i z_i^2)^r$$ to the set of sum of squares polynomials. This amounts to solving a semidefinite program. We remark that all semidefinite programming-based hierarchies available only produce an approximate solution to the optimal value of the SDP solved at level $r$ in polynomial time. This is independent of whether they use bisection (e.g., such as the hierarchy given in Theorem \ref{th:reznick.hierarchy} or the one based on Stengle's Positivstellensatz) or not (e.g., the Lasserre hierarchy).
	\end{remark}

Our next theorem improves on our previous hierarchy by freeing the multiplier $(\sum_{i=1}^N z_i^2)^r$ and taking advantage of our ability to search for an optimal multiplier using semidefinite programming.

\begin{theorem}\label{th:artin.hierarchy} Recall the definition of Artin cones from Definition \ref{def:reznick.artin.cones}. Consider the hierarchy of optimization problems indexed by $r$:
	\begin{equation}\label{eq:artin.hierarchy}
	\begin{aligned}
	l_r\mathrel{\mathop{:}}=&\sup_{\gamma,q} &&\gamma\\
	&\text{s.t. } &&f_{\gamma}(z)-\frac{1}{r}(\sum_{i=1}^N z_i^2)^{D} \in A_{N,2D}^r.
	\end{aligned}
	\end{equation}
	Then, $l_r \leq p^*$ for all $r$, $\{l_r\}$ is nondecreasing, and $\lim_{r \rightarrow \infty} l_r=p^*.$
\end{theorem}

\begin{proof}
	Just as the previous theorem, it suffices to show that the Artin cones $A_{n,2d}^r$ satisfy properties (a)-(d) of Theorem \ref{th:hierarchy}. The proof of property (a) follows the proof given for Theorem \ref{th:reznick.hierarchy}. Property (b) is satisfied as a (weaker) consequence of Artin's result (see Theorem \ref{th:artin}). For (c), we have that if $p(x) \cdot q(x)$ is sos for some {\gh nonzero} sos polynomial {\gh $q$} of degree $2r$, then $p(x) \cdot q(x) \cdot (\sum_i x_i^2)$ is sos, and $q(x)\cdot (\sum_i x_i^2)$ has degree $2(r+1)$. Finally, for (d), suppose that $p \in A_{n,2d}^r$. Then there exists an sos form $q$ such that $p(x) \cdot q(x)$ is sos. We have $$\left(p(x)+\epsilon (\sum_i x_i^2)^{d}\right) \cdot q(x)=p(x)\cdot q(x)+\epsilon (\sum_i x_i^2)^{d} \cdot q(x),$$ which is sos as the product (resp. sum) of two sos polynomials is sos.
	\end{proof}

Note that again, for any fixed $r$, the level $r$ of the hierarchy can be solved using bisection which leads to a sequence of semidefinite programs.

Our developments in the past two sections can be phrased in terms of a Positivstellensatz.

\begin{corollary}[A new Positivstellensatz]\label{cor:SDP.psatz}
	Consider the basic semialgebraic set $$S\mathrel{\mathop{:}}=\{x \in \mathbb{R}^n~|~ g_i(x)\geq 0, i=1,\ldots,m\}$$ and a polynomial $p\mathrel{\mathop{:}}=p(x)$. Suppose that $S$ is contained within a ball of radius $R$. Let $\eta_i$ and $\beta$ be any finite upperbounds on $g_i(x)$ and, respectively, $-p(x)$ over the set $S$.\footnote{As discussed at the beginning of Section \ref{sec:nonneg.approx}, such bounds are very easily computable.} Let $d$ be such that $2d$ is the smallest integer larger than or equal to the maximum degree of $p, g_i, i=1,\ldots,m$. Then, $p(x)>0$ for all $x \in S$ if and only if there exists a positive integer $r$ such that
	$$\left(h(x,s,y)-\frac{1}{r} \left(\sum_{i=1}^n x_i^2+\sum_{j =0}^{m+1} s_j^2+y^2\right)^{2d}\right) \cdot \left(\sum_{i=1}^n x_i^2+\sum_{j =0}^{m+1} s_j^2+y^2\right)^{r}$$ is a sum of squares, where the form $h$ in variables $(x_1,\ldots,x_n, s_0,\ldots,s_{m+1},y)$ is as follows:
	 \begin{align*}
	h(x,s,y)\mathrel{\mathop{:}}=&\left(y^{2d}p(x/y)+s_0^2y^{2d-2} \right)^2+\sum_{i=1}^m\left(y^{2d}g_i(x/y)-s_i^2 y^{2d-2}\right)^2\\
	&+\left((R+\sum_{i=1}^m \eta_i +\beta)^dy^{2d}-(\sum_{i=1}^n x_i^2+\sum_{i=0}^m s_i^2)^d-s_{m+1}^{2d}\right)^2.
	\end{align*}

\end{corollary} 

\begin{proof}
	This is an immediate corollary of arguments given in the proof of Theorem \ref{th:slb} and in the proof of Theorem \ref{th:reznick.hierarchy} for the case where $\gamma=0.$
\end{proof}

\section{Poly\'a's theorem and hierarchies for POPs that are optimization-free, LP-based, and SOCP-based}\label{sec:opt.free.LP.SOCP}

In this section, we use a result by Poly\'a on global positivity of even forms to obtain new hierarchies for polynomial optimization problems. In Section~\ref{subsec:opt.free}, we present a hierarchy that is \emph{optimization-free}, in the sense that each level of the hierarchy only requires multiplication of two polynomials and checking if the coefficients of the resulting polynomial are nonnegative. In Section~\ref{subsec:lp.socp}, we use the previous hierarchy to derive linear programming and second-order cone programming-based \aaa{(converging) hierarchies that in each level produce a lower bound on the POP whose quality is at least as good as that of the optimization-free hierarchy}. These rely on the recently developed concepts of dsos and sdsos polynomials (see Definition \ref{def:dsos.sdsos} and \cite{iSOS_journal}), which are alternatives to sos polynomials that have been used in diverse applications to improve scalability; see \cite[Section 4]{iSOS_journal}.

\subsection{An optimization-free hierarchy of lower bounds for POPs}\label{subsec:opt.free}

The main theorem in this section presents an optimization-free hierarchy of lower bounds for general POPs with compact feasible sets:

\begin{theorem}\label{th:sms.hierarchy}
	Recall the definition of $f_{\gamma}(z)$ as given in (\ref{eq:f.gamma}), with $z \in \mathbb{R}^N$ and $deg(f_{\gamma})=2D.$ Let $(v,w) \in \mathbb{R}^{2n}$ and define	{\gh
	\begin{equation}\label{eq:def.Knd}
	\begin{aligned}
	&Pol_{n,2d}^r\mathrel{\mathop{:}}=&&\{p \in H_{n,2d}~|~\left(p(v^2-w^2)+ \frac{1}{2r}\left(\sum_{i=1}^n (v_i^4+w_i^4)\right)^d\right) \cdot \left(\sum_i v_i^2+\sum_i w_i^2\right)^{r^2} \phantom\}\\
	& &&\phantom\{ \text{ has nonnegative coefficients } \}.
	\end{aligned}
	\end{equation}
}
	
	Consider the hierarchy of optimization problems indexed by $r$:
	\begin{equation}\label{eq:sms.hierarchy}
	\begin{aligned}
	l_r\mathrel{\mathop{:}}=&\sup_{\gamma} &&\gamma\\
	&\text{s.t. } && f_{\gamma}(z)-\frac{1}{r}(\sum_{i=1}^N z_i^2)^D \in Pol_{N,2D}^r.
	\end{aligned}
	\end{equation}
	Let $m_r=\max_{i=1,\ldots,r} l_i$. Then $m_r \leq p^*$ for all $r$, $\{m_r\}$ is nondecreasing, and $\lim_{r \rightarrow \infty} m_r=p^*$.
\end{theorem}

As before, we use bisection to obtain the optimal value $l_r$ of the $r^{th}$ level of the hierarchy up to a fixed precision $\epsilon$ (see Remark \ref{rem:bisection}). At each step of the bisection algorithm, one simply needs to multiply two polynomials together and check nonnegativity of the coefficients of the resulting polynomial to proceed to the next step. As a consequence, this hierarchy is optimization-free as we do not need to solve (convex) optimization problems at each step of the bisection algorithm. To the best of our knowledge, no other converging hierarchy of lower bounds for {\gh general POPs (whose feasible sets are contained within a ball of known radius)} dispenses altogether with the need to solve convex subprograms. We also provide a Positivstellensatz counterpart to the hierarchy given above (see Corollary \ref{cor:Positiv.LP}). This corollary implies in particular that one can always certify infeasibility of a basic semialgebraic set by recursively multiplying polynomials together and simply checking nonnegativity of the coefficients of the resulting polynomial.

We now make a few remarks regarding the techniques used in the proof of Theorem~\ref{th:sms.hierarchy}. Unlike Theorems \ref{th:reznick.hierarchy} and \ref{th:artin.hierarchy}, we do not show that $Pol_{n,d}^r$ satisfies properties (a)-(d) as given in Theorem~\ref{th:hierarchy} due to some technical difficulties. It turns out however that we can avoid showing properties (c) and (d) by using a result by Reznick and Powers \cite{PowersReznick} that we present below. Regarding properties (a) and (b), we show that a slightly modified version of (a) holds and that (b), which is the key property in Theorem \ref{th:hierarchy}, goes through as is. We note though that obtaining (b) from Poly\'a's result (Theorem \ref{th:polya}) is not as immediate as obtaining (b) from Artin's and Reznick's results. Indeed, unlike the theorems by Artin and Reznick (see Theorems \ref{th:artin} and \ref{th:reznick.uniform}) which certify global positivity of \emph{any} form, Poly\'a's result only certifies global positivity of \emph{even} forms. To make this latter result a statement about general forms, we work in an appropriate lifted space. {\gh We make the simple observation that any scalar $x$ can be written as $x=v^2-w^2$ with $vw=0$ (take $v=
	\sqrt{\max\{x,0\}}$ and $w=\sqrt{\max\{-x,0\}}$). We then replace the form} $p(z)$ in variables $z \in \mathbb{R}^n$ by the even form $p(v^2-w^2)$ in variables $(v,w) \in \mathbb{R}^{2n}$. This lifting operation preserves nonnegativity, but unfortunately it does not preserve positivity: even if $p(z)$ is pd, $p(v^2-w^2)$ always has zeros (e.g., when $v=w$). Hence, though we now have access to an even form, we still cannot use Poly\'a's property as $p(v^2-w^2)$ is not positive. This is what leads us to consider the slightly more complicated form $p(v^2-w^2)+\frac{1}{2r}(\sum_i v_i^4+w_i^4)^d$ in (\ref{eq:def.Knd}). 

\begin{theorem}[Powers and Reznick \cite{PowersReznick}]\label{th:reznick}
Let $\alpha=(\alpha_1,\ldots,\alpha_n) \in \mathbb{N}^n$, $x^\alpha=x_1^{\alpha_1}\ldots x_n^{\alpha_n}$, and write $|\alpha|=\alpha_1+\ldots+\alpha_n.$ Denote the standard simplex by $\Delta_n$, {\gh i.e., $\Delta_n=\{(x_1,\ldots,x_n) \in \mathbb{R}^n|~x_i\geq 0, i=1,\ldots,n, \text{ and } x_1+\ldots+x_n=1\}$}. Assume that $f$ is a form of degree $2d$ that is positive on $\Delta_n$ and let $$\lambda=\lambda(f)\mathrel{\mathop{:}}=\min_{x \in \Delta_n} f(x).$$
Define $c(\alpha)=\frac{(2d)!}{\alpha_1! \ldots \alpha_n!}.$ We have:
$$f(x)=\sum_{|\alpha|=2d} a_{\alpha} x^{\alpha}=\sum_{|\alpha|=2d} b_{\alpha}c(\alpha)x^{\alpha}.$$
Let \aaa{$||f||\mathrel{\mathop{:}}=\max_{|\alpha|=2d} |b_{\alpha}|$.}\footnote{As defined, $||f||$ is a submultiplicative norm; see \cite{schweighofer2004complexity}.}

Then, the coefficients of $$f(x_1,\ldots,x_n) \cdot (x_1+\ldots+x_n)^{\bar{N}}$$ are nonnegative for \aaa{$\bar{N}> d(2d-1)\frac{||f||}{\lambda}-2d$.}
\end{theorem}

Note that here the bound is given in the case where one considers the alternative (but equivalent) formulation of Poly\'a's Positivstellensatz to the one given in Theorem \ref{th:polya}, i.e., when one is concerned with positivity of a form over the simplex. The result can easily be adapted to the formulation where one considers global positivity of an even form as shown below.

\begin{lemma}\label{th:Reznick.even}
Let $p\mathrel{\mathop{:}}=p(x)$ be an even form of degree $2d$ that is positive definite and let $\beta>0$ be its minimum on $S_x$. \aaa{Define $q(x_1,\ldots,x_n)\mathrel{\mathop{:}}=p(\sqrt{x_1},\ldots,\sqrt{x_n})$.} Then, $$p(x_1,\ldots,x_n) \cdot (\sum_i x_i^2)^{\bar{N}}$$ has nonnegative coefficients for \aaa{$\bar{N}>d(2d-1)\frac{||q||}{\beta}-2d$.}
\end{lemma}

\begin{proof}
\aaa{Recall that $q(x_1,\ldots,x_n)=p(\sqrt{x_1},\ldots,\sqrt{x_n})$.} Since $p(x)\geq \beta$ on $S_x$, then \aaa{$q(x) \geq \beta$} on $\Delta_n.$ Indeed, by contradiction, suppose that there exists $\hat{x} \in \Delta_n$ such that \aaa{$q(\hat{x}) =\beta -\epsilon$} (where $\epsilon>0$) and let $y=\sqrt{\hat{x}}$. Note that as $\sum_i \hat{x}_i=1$, we have $\sum_i y_i^2=1$. Furthermore, \aaa{$p(y)=q(\hat{x})=\beta-\epsilon$} which contradicts the assumption. Hence, using Theorem~\ref{th:reznick}, we have that when \aaa{$\bar{N}>d(2d-1)\frac{||q||}{\beta}-2d$, $$q(x)(\sum_i x_i)^{\bar{N}} $$}
	has nonnegative coefficients. Hence, \aaa{
	$$q(y^2)(\sum_i y_i^2)^{\bar{N}} =p(y)(\sum_i y_i^2)^{\bar{N}}$$}
	also has nonnegative coefficients.
\end{proof}

Before we proceed with the proof of Theorem \ref{th:sms.hierarchy}, we need the following lemma.

\begin{lemma} \label{lem:outstrip}
	Let 	{\gh
	\begin{equation}\label{eq:def.p.gamma}
	p_{\gamma,r}(v,w)\mathrel{\mathop{:}}=f_\gamma(v^2-w^2)-\frac{1}{r}\left(\sum_{i=1}^N \left(v_i^2-w_i^2\right)^2\right)^D+\frac{1}{2r} \left(\sum_{i=1}^N (v_i^4+w_i^4)\right)^D,
	\end{equation}
}
	where $f_{\gamma}$ is defined as in (\ref{eq:f.gamma}), \aaa{let $q_{\gamma,r}(v,w)\mathrel{\mathop{:}}=p_{\gamma,r} (\sqrt{v}, \sqrt{w})$} and let \aaa{$$\bar{N}(r)=D(2D-1) \cdot \frac{||q_{\gamma,r}||}{\min_{S_{v,w}}p_{\gamma,r}(v,w)}-2D.$$} If $f_{\gamma}(z)$ is positive definite, there exists $\hat{r}$ such that $r^2 \geq \bar{N}(r)$,  for all $r \geq \hat{r}$.
	\end{lemma}

\begin{proof}
	As $f_{\gamma}(z)$ is positive definite, there exists a positive integer $r_0$ such that $f_{\gamma}(z)-\frac{1}{r}(\sum_i z_i^2)^D$ is positive definite for all $r \geq r_0$ and hence 
	\begin{equation}\label{eq:nonneg.proof}
	f_{\gamma}(v^2-w^2)-\frac{1}{r} (\sum_i (v_i^2-w_i^2)^2)^D
	\end{equation}
	 is nonnegative for all $r \geq r_0$. Recall now that $||x||_p=(\sum_{i=1}^n x_i^p)^{1/p}$ is a norm for $p \geq 1$ and that $$||x||_2 \leq n^{1/4}||x||_4.$$ This implies that $$(\sum_i v_i^4+\sum_i w_i^4)^D \geq \frac{1}{(2N)^{D}} (\sum_i v_i^2+\sum_i w_i^2)^{2D}$$ and hence in view of (\ref{eq:nonneg.proof}) and the definition of $p_{\gamma,r}$, we have $$p_{\gamma,r}(v,w) \geq \frac{1}{2^{D+1}N^{D}r} (\sum_i v_i^2+\sum_{i} w_i^2)^{2D}, \forall r\geq r_0.$$
	This enables us to conclude that 
	\begin{align}\label{eq:min.p.gamma}
	\min_{S_{v,w}} p_{\gamma,r}(v,w) \geq \frac{1}{2^{D+1}N^{D}r}, \text{ for any } r \geq r_0.
	\end{align}
\aaa{Further, notice that if we define $\bar{f}_\gamma(v,w):=f_\gamma(v-w)$, then using properties of the norm, we have the following chain of inequalities for any positive integer $r$:
\begin{align*}
 ||q_{\gamma,r}||&\leq ||\bar{f}_\gamma||+\frac{1}{r}||(\sum_i (v_i-w_i)^2)^D||+\frac{1}{2r} ||(\sum_i (v_i^2+w_i^2))^D||\\
&\leq ||\bar{f}_\gamma||+||(\sum_i (v_i-w_i)^2)^D||+ ||(\sum_i v_i^2+w_i^2)^D||=\mathrel{\mathop{:}}c_{\gamma}.
\end{align*}}
As a consequence, combining this with the definition of $\bar{N}(r)$ and (\ref{eq:min.p.gamma}), we have $$\bar{N}(r) \leq D(2D-1)2^{{D+1}}rN^{D}c_{\gamma}, ~\forall r \geq r_0.$$
Now taking $\hat{r}=\max(r_0,\lceil D(2D-1)2^{{D+1}}N^{D}c_{\gamma}\rceil)$, we have $r^2 \geq \bar{N}(r), \forall r\geq \hat{r}.$
	\end{proof}
We now proceed with the proof of Theorem \ref{th:sms.hierarchy}.

\begin{proof}[Proof of Theorem \ref{th:sms.hierarchy}]
	By definition, the sequence $\{m_r\}$ is nondecreasing. We show that it is upperbounded by $p^*$ by showing that if $\gamma$ is such that $$f_{\gamma}(z)-\frac{1}{r}(\sum_i z_i^2)^D \in Pol_{N,2D}^r,$$ for some $r$, then $f_{\gamma}$ must be positive definite. Then Theorem \ref{th:slb} gives us that $\gamma$ is a strict lower bound on (\ref{eq:POP}). As $p^* > \gamma$ for any such $\gamma$, we have that $l_r\leq p^*,\forall r$ and hence $m_r \leq p^*, \forall r.$ 
	
	Assume that $\gamma$ is such that
	$$f_{\gamma}(z)-\frac{1}{r}(\sum_i z_i^2)^D \in Pol_{N,2D}^r$$ for some $r$.
By definition of $Pol_{N,2D}^r$ and as $(\sum_i v_i^2+\sum_i w_i^2)^{r^2}$ is nonnegative, we get that the form
$$f_\gamma(v^2-w^2)-\frac{1}{r}(\sum_i (v_i^2-w_i^2)^2)^D+\frac{1}{2r} (\sum_i v_i^4+w_i^4)^D$$ is nonnegative.
This implies that 
	\begin{align}\label{eq:ineq.tv.screen}
f_{\gamma}(v^2-w^2)-\frac{1}{r}(\sum_i (v_i^2-w_i^2)^2)^D \geq -\frac{1}{2r} \text{ for } (v,w) \in \{(v,w)~|~ \sum_i v_i^4+\sum_i w_i^4 =1\},
	\end{align} which gives 
	\begin{align}\label{eq:ineq.sphere}
		f_{\gamma}(z)-\frac{1}{r}(\sum_i z_i^2)^D \geq -\frac{1}{2r},~ \forall z\in S_z.
	\end{align}
 Indeed, suppose that there exists $\hat{z} \in S_z$ such that (\ref{eq:ineq.sphere}) does not hold. Then, let $\hat{z}^+=\max(\hat{z},0)$ and $\hat{z}^-=\max(-\hat{z},0)$. Note that both $\hat{z}^+$ and $\hat{z}^-$ are nonnegative so we can take $\hat{v}=\sqrt{\hat{z}^+}$ and $\hat{w}=\sqrt{\hat{z}^-}.$ We further have that as $\hat{z} \in S_z$ and $\hat{z}=\hat{v}^2-\hat{w}^2$, $\sum_i \hat{v}_i^4+\sum_i \hat{w}_i^4=1$. Substituting $\hat{z}$ by $\hat{v}^2-\hat{w}^2$ in (\ref{eq:ineq.sphere}) then violates (\ref{eq:ineq.tv.screen}). Using (\ref{eq:ineq.sphere}), we conclude that $$f_{\gamma}(z)\geq \frac{1}{2r},~\forall z\in S_z$$ and that $f_{\gamma}$ is positive definite. 
 
 We now show that the hierarchy converges, i.e., that $\lim_{r \rightarrow \infty} m_r=p^*$. To do this, we show that if $\gamma$ is a strict lower bound on (\ref{eq:POP}), or equivalently from Theorem \ref{th:slb}, if $f_{\gamma}(z)$ is positive definite, then there exists $r'$ such that $$f_{\gamma}(z)-\frac{1}{r'} (\sum_{i}z_i^2)^D \in Pol_{N,2D}^{r'}.$$ 
 Since $f_{\gamma}$ is pd, there exists a positive integer $r_0$ such that $f_{\gamma}(z)-\frac{1}{r}(\sum_{i=1}^N z_i^2)^D$ is pd for any $r \geq r_0$. This implies that $f_{\gamma}(v^2-w^2)-\frac{1}{r}(\sum_i (v_i^2-w_i^2)^2)^D$ is nonnegative and $$f_\gamma(v^2-w^2)-\frac{1}{r}(\sum_i (v_i^2-w_i^2)^2)^D+\frac{1}{2r} (\sum_i(v_i^4+w_i^4))^D$$ is positive definite for $r \geq r_0$. {\gh Since this form is even, using Lemma \ref{th:Reznick.even} and the definition of $\bar{N}(r)$ in Lemma \ref{lem:outstrip}, for any $r\geq r_0$, we have that the polynomial $$\left(f_\gamma(v^2-w^2)-\frac{1}{r}(\sum_i (v_i^2-w_i^2)^2)^D+\frac{1}{2r} (\sum_i (v_i^4+w_i^4))^D \right) \cdot (\sum_i v_i^2+\sum_i w_i^2)^{\lceil \bar{N}(r)\rceil}$$ has nonnegative coefficients}. From Lemma \ref{lem:outstrip}, there exists $\hat{r}$ such that $r \geq \hat{r}$ implies $r^2 \geq \bar{N}(r).$ Taking $r'=\max\{r_0,\hat{r}\}$ and considering $p_{\gamma,r'}$ as defined in (\ref{eq:def.p.gamma}), we get that
 \begin{align*}
 &p_{\gamma,r'}(v,w) (\sum_i v_i^2+\sum_i w_i^2)^{r'^2} \\ &=p_{\gamma,r'}(v,w)(\sum_i v_i^2+\sum_i w_i^2)^{\lceil \bar{N}(r') \rceil} \cdot (\sum_i v_i^2+\sum_i w_i^2)^{r'^2-\lceil \bar{N}(r') \rceil} 
 \end{align*}
 has nonnegative coefficients, which is the desired result. This is because $$p_{\gamma,r'}(v,w)(\sum_i v_i^2+\sum_i w_i^2)^{\lceil \bar{N}(r') \rceil}$$ has nonnegative coefficients as $r' \geq r_0$, and $$(\sum_i v_i^2+\sum_i w_i^2)^{r'^2-\lceil \bar{N}(r') \rceil}$$ has nonnegative coefficients as $r' \geq \hat{r}$, and that the product of two polynomials with nonnegative coefficients has nonnegative coefficients.
\end{proof}

\begin{corollary}[An optimization-free Positivstellensatz]\label{cor:Positiv.LP} Consider the {\gh closed} basic semialgebraic set $$S\mathrel{\mathop{:}}=\{x \in \mathbb{R}^n~|~ g_i(x)\geq 0, i=1,\ldots,m\}$$ and a polynomial $p\mathrel{\mathop{:}}=p(x)$. Suppose that $S$ is contained within a ball of radius $R$. Let $\eta_i$ and $\beta$ be any finite upperbounds on $g_i(x)$ and, respectively, $-p(x)$ over the set $S$.\footnote{Once again, as discussed at the beginning of Section \ref{sec:nonneg.approx}, such bounds are very easily computable.} Let $d$ be such that $2d$ is the smallest even integer larger than or equal to the maximum degree of $p, g_i, i=1,\ldots,m$.  Then, $p(x)>0$ for all $x \in S$ if and only if there exists a positive integer $r$ such that	
	\begin{align*}
	\left(h(v^2-w^2)-\frac{1}{r}(\sum_{i=1}^{n+m+3} (v_i^2-w_i^2)^2)^d+\frac{1}{2r} (\sum_{i=1}^{n+m+3} (v_i^4+w_i^4))^d\right) \\
	\cdot \left(\sum_{i=1}^{n+m+3}v_i^2+\sum_{i=1}^{{n+m+3}} w_i^2\right)^{r^2}
	\end{align*}
	has nonnegative coefficients, where the form $h \mathrel{\mathop{:}}=h(z)$ in variables $$(z_1,\ldots,z_{n+m+3})\mathrel{\mathop{:}}=(x_1,\ldots,x_n,s_0,\ldots,s_{m+1},y)$$ is as follows:
	\begin{align*}
	h(x,s,y)\mathrel{\mathop{:}}=&\left(y^{2d}p(x/y)+s_0^2y^{2d-2} \right)^2+\sum_{i=1}^m\left(y^{2d}g_i(x/y)-s_i^2 y^{2d-2}\right)^2\\
	&+\left((R+\sum_{i=1}^m \eta_i +\beta)^dy^{2d}-(\sum_{i=1}^n x_i^2+\sum_{i=0}^m s_i^2)^d-s_{m+1}^{2d}\right)^2.
	\end{align*}
\end{corollary}

\begin{proof} 
	This is an immediate corollary of arguments given in the proof of Theorem~\ref{th:slb} and in the proof of Theorem \ref{th:sms.hierarchy} for the case where $\gamma=0.$
	\end{proof}

\subsection{Linear programming and second-order cone programming-based hierarchies for POPs}\label{subsec:lp.socp}

In this section, we present a linear programming and a second-order cone programming-based hierarchy for general POPs which by construction converge faster than the hierarchy presented in Section~\ref{subsec:opt.free}. These hierarchies are based on the recently-introduced concepts of dsos and sdsos polynomials \cite{iSOS_journal} which we briefly revisit below to keep the presentation self-contained.

\begin{definition}\label{def:dd.sdd}
	A symmetric matrix $M$ is said to be
	\begin{itemize}
		\item \emph{diagonally dominant (dd)} if $M_{ii} \geq \sum_{j \neq i}|M_{ij}|$ for all $i$. 
		\item  \emph{scaled diagonally dominant (sdd)} if there exists a diagonal matrix $D,$ with positive diagonal entries, such that {\gh$DMD$} is dd.
	\end{itemize} 
\end{definition}
We have the following implications as a consequence of Gershgorin's circle theorem:
\begin{align}\label{eq:implic.matrices}
\text{M } dd \Rightarrow \text{M } sdd \Rightarrow \text{M } \succeq 0.
\end{align}
Requiring $M$ to be dd (resp. sdd) can be encoded via a linear program (resp. a second-order cone program) (see \cite{iSOS_journal} for more details). These notions give rise to the concepts of dsos and sdsos polynomials.

\begin{definition}[\cite{iSOS_journal}]\label{def:dsos.sdsos}
	Let $z(x)=(x_1^d,x_1^{d-1}x_2,\ldots,x_n^d)^T$ be the vector of monomials in $(x_1,\ldots,x_n)$ of degree $d$. A form $p \in H_{n,2d}$ is said to be 
	\begin{itemize}
		\item \emph{diagonally-dominant-sum-of-squares (dsos)} if it admits a representation$$p(x)=z^T(x)Qz(x), \text{ where $Q$ is a dd matrix.}$$
		\item \emph{scaled-diagonally-dominant-sum-of-squares (sdsos)} if it admits a representation$$p(x)=z^T(x)Qz(x), \text{ where $Q$ is a sdd matrix.}$$
	\end{itemize}
\end{definition}

The following implications are a consequence of (\ref{eq:implic.matrices}):
\begin{align}\label{eq:implicsdsos}
p(x) \text{ dsos} \Rightarrow p(x) \text{ sdsos} \Rightarrow p(x) \text{ sos} \Rightarrow p(x) \text{ nonnegative}.
\end{align}
Given the fact that our Gram matrices and polynomials are related to each other via linear equalities, it should be clear that optimizing over the set of dsos (resp. sdsos) polynomials is an LP (resp. SOCP).

We now present our LP and SOCP-based hierarchies for POPs.

\begin{corollary}\label{cor:dsos.sdsos.hierarchy}
	
Recall the definition of $f_{\gamma}(z)$ as given in (\ref{eq:f.gamma}), with $z \in \mathbb{R}^N$ and $deg(f_{\gamma})=2D$, and let $p_{\gamma,r}$ be as in (\ref{eq:def.p.gamma}). Consider the hierarchy of optimization problems indexed by $r$:
	\begin{equation}\label{eq:dsos.sdsos.hierarchy}
	\begin{aligned}
	l_r\mathrel{\mathop{:}}=&\sup_{\gamma,q} &&\gamma\\
	&\text{s.t. } && p_{\gamma,r}(v,w) \cdot q(v,w) \text{ is s/dsos}\\
	& &&q(v,w) \text{ is s/dsos and of degree $2r^2$}.
	\end{aligned}
	\end{equation}
	Let $m_r=\max_{i=1,\ldots,r} l_i$. Then, $m_r \leq p^*$ for all $r$, $\{m_r\}$ is nondecreasing, and we have $\lim_{r \rightarrow \infty} m_r=p^*$.
\end{corollary}

\begin{proof}
	This is an immediate consequence of the fact that any even form $p \in H_{n,2d}$ with nonnegative coefficients can be written as $p(x)=z(x)^TQz(x)$ where $Q$ is diagonal and has nonnegative (diagonal) entries. As such a $Q$ is  dd (and also sdd), we conclude that $p$ is dsos (and also sdsos). The corollary then follows from Theorem~\ref{th:sms.hierarchy}.
	\end{proof}

Note that similarly to our previous hierarchies, one must proceed by bisection on $\gamma$ to solve the level $r$ of the hierarchy. At each step of the hierarchy, we solve a linear program (resp. second-order cone program) that searches for the coefficients of $q$ that make $q$ dsos (resp. sdsos) and $p_{\gamma,r} \cdot q$ dsos (resp. sdsos).

There is a trade-off between the hierarchies developed in this subsection and the one developed in the previous subsection: the hierarchy of Section \ref{subsec:opt.free} is optimization-free whereas those of Section \ref{subsec:lp.socp} use linear or second-order cone programming. Hence the former hierarchy is faster to run at each step. However, the latter hierarchies could potentially take fewer levels to converge. This is similar to the trade-off observed between the hierarchies presented in Theorems \ref{th:reznick.hierarchy} and \ref{th:artin.hierarchy}.

\section{Future research directions and open problems}\label{sec:conclusion}
{\gh To conclude, we present some interesting directions for future research and some open problems.	

A first research direction of interest relates to the computational performance of the methods described in this paper. As they stand, we do not believe that our hierarchies are as practically efficient as other hierarchies (e.g., those due to Lasserre and Parrilo). This is mainly because all our hierarchies \aaa{require the use of bisection and increase  the number of variables (from $n$ to $n+m+3$) and the degree (from $2d$ to $4d$) of the polynomials involved}. Though this blow-up is linear, one can expect that it could be problematic in practice. Out of our hierarchies, the optimization-free one may be tempting to consider for very large-scale problems as each step of it only requires polynomial multiplication.
%Because each step of the optimization-free hierarchy only requires polynomial multiplication, it is tempting to consider it for very large-scale problems. 
However, the caveat that one should keep in mind is that the level $r$ required to obtain a good-quality lower bound to the POP would probably be very large. To counter-balance this, one would have to invest considerable effort in a proper implementation of the arithmetic involved in each level of the hierarchy. Such an implementation should involve among other things: (i) automation of the computation of the coefficients of $f_{\gamma}$, (ii) exploiting the very specific structure of the polynomial in (\ref{eq:def.Knd}) (with $p$ replaced by $f_\gamma(z)-\frac{1}{r}(\sum_{i=1}^N z_i^2)^D $) to compute its coefficients explicitly, (iii) efficient multiplication of two multivariate polynomials (e.g., by evaluating the product on random samples and obtaining its coefficients by solving a linear system, or by pursuing the ideas in~\cite{johnson1974sparse}), and (iv) thinking about ways of parallelizing all these computations. We are unsure as to whether the optimization-free hierarchy would provide useful bounds on relevant examples---even after careful implementation---but it may be an interesting research direction to investigate.

Another interesting research direction lies in the comparison of our hierarchies to existing ones. One way to do this would be to determine the convergence rates of the hierarchies presented in this paper and compare them with those of other hierarchies. The main difficulty in obtaining such rates is the absence of a quantitative version of Theorem \ref{th:slb}, which would link the quality of the lower bound on the optimal value of the POP (i.e., the value of the gap $p^*-\gamma$) to the minimum of $f_{\gamma}$ on the unit sphere in $z$-space. If such a result were available, one could have hopes, e.g., to obtain convergence rates for our optimization-free hierarchy by applying Theorem \ref{th:reznick} and Lemma \ref{th:Reznick.even} to $f_{\gamma}$. (A similar approach was undertaken by de Klerk, Laurent, and Parrilo in \cite{de2006ptas} to provide convergence rates for an analogous hierarchy for the problem of minimizing a polynomial on the simplex.) Likewise, one could hope to relate Lasserre's hierarchy (for example) to the Reznick or Artin hierarchies presented in this paper and derive convergence rates for them by using previously known results on the complexity of the Lasserre hierarchy \cite{nie2007Putinarcomplexity}. As it stands, we are unable to show that if level $r$ of Lasserre's hierarchy certifies positivity of $p(x)-\gamma$ on the feasible set, then so does a level $r'$ of our Reznick or Artin hierarchies. This is due to the fact that, from a sum of squares certificate of positivity of $p(x)-\gamma$, one does not a priori know how to pick $r'$ in such a way that $f_{\gamma}(z)-\frac{1}{r'}(\sum_i z_i^2)^{2d}$ becomes positive definite (let alone admits a particular certificate of positivity). This question would again be answered if we had a quantitative version of Theorem \ref{th:slb}.
%
%
%
%If level $r$ of Lasserre
%
%
%one could attempt to understand  
%
%
%
%
%One way to do this would be to determine the convergence rate of the hierarchies presented in this paper and compare them with those of existing hierarchies. In particular, one could expect that the convergence rate of our ``optimization-free'' hierarchy may be easier to analyze as it is derived from Poly\'a's Positivstellensatz. This is the case, e.g., for hierarchies of lower bounds on polynomial optimization problems over the simplex \cite{de2006ptas}, which also rely on Poly\'a's Positivstellensatz. To obtain a convergence rate result, one would most likely have to apply Theorem \ref{th:reznick} and Lemma \ref{th:Reznick.even} to $f_{\gamma}$. However before this can be carried out, one would need to establish a quantitative version of Theorem \ref{th:slb}, which would link the quality of the lower bound (i.e., the value of the gap $|p^*-\gamma|$) to positivity of $f_{\gamma}$ in an explicit way. Thus far, this is something that we have been unable to do.

Finally, we present two more concrete open problems spawned by the writing of this paper. The first one concerns the assumptions needed to construct our hierarchies.}

\paragraph{Open problem 1} Theorems \ref{th:slb} and \ref{th:hierarchy} require that the feasible set $S$ of the POP given in (\ref{eq:POP}) be contained in a ball of radius $R$. Can these theorems be extended to the case where there is no compactness assumption on $S$?

\indent The second open problem is linked to the Artin and Reznick cones presented in Definition \ref{def:reznick.artin.cones}. 
\paragraph{Open problem 2} As mentioned before, Reznick cones $R_{n,2d}^r$ are convex for all $r$. We are unable to prove however that Artin cones $A_{n,2d}^r$ are convex (even though they satisfy properties (a)-(d) of Theorem \ref{th:hierarchy} like Reznick cones do). Are Artin cones convex for all $r$? We know that they are convex for $r=0$ and for $r$ large enough as they give respectively the sos and psd cones (see \cite{lombardi2014elementary} for the latter claim). However, we do not know the answer already for $r=1$. {\gh Our preliminary attempts to find forms that would disprove convexity of these cones have failed so far.}

\vspace{2mm}
\section*{Acknowledgments} We are grateful to Pablo Parrilo for very insightful comments, particularly as regards Section \ref{sec:opt.free.LP.SOCP} and the observation that any form can be made even by only doubling the number of variables and the degree. We would also like to thank Tin Nguyen for suggesting reference~\cite{johnson1974sparse}. Finally, we are very grateful to two anonymous referees and an anonymous associate editor for several constructive comments which have greatly improved this manuscript.

\bibliographystyle{siamplain}
\bibliography{pablo_amirali}
\end{document}